\documentclass[reqno]{amsart}
\usepackage{amsmath, amssymb, amsthm, epsfig}
\usepackage{hyperref, latexsym}
\usepackage{url}
\usepackage[mathscr]{euscript}

\usepackage{color}
\usepackage{fullpage} 
\usepackage{setspace}
\usepackage{caption}
\usepackage{soul}
\usepackage{xcolor}

\def\today{\ifcase\month\or
  January\or February\or March\or April\or May\or June\or
  July\or August\or September\or October\or November\or December\fi
  \space\number\day, \number\year}

 \newtheorem{theorem}{Theorem}
 \newtheorem{conjecture}[theorem]{Conjecture}
 \newtheorem{lemma}[theorem]{Lemma}
 
 \newtheorem{corollary}[theorem]{Corollary}
 \theoremstyle{definition}
 \newtheorem{definition}[theorem]{Definition}
 
 \theoremstyle{remark}

\renewcommand{\S}{\mathbb{S}}

 \newcommand{\C}{\mathbb{C}}
 \newcommand{\R}{\mathbb{R}}
 \newcommand{\N}{\mathbb{N}}
 
 \newcommand{\Z}{\mathbb{Z}}

 \newcommand{\du}{\text{\rm d}u}

 \newcommand{\dx}{\text{\rm d}x}

  \newcommand{\dr}{\text{\rm d}r}

    \renewcommand{\d}{\text{\rm d}}

\newcommand{\ov}{\overline}

\begin{document}

\title[]{A sharp trilinear inequality related to \\ Fourier restriction on the circle}
\author[Carneiro, Foschi, Oliveira e Silva, Thiele]{Emanuel Carneiro, Damiano Foschi, Diogo Oliveira e Silva and Christoph Thiele}
\date{\today}
\address{IMPA - Instituto Nacional de Matem\'{a}tica Pura e Aplicada, Estrada Dona Castorina 110, Rio de Janeiro, RJ 22460-320, Brazil}
\email{carneiro@impa.br}
\address{Dipartimento di Matematica e Informatica, Universit\`{a} di Ferrara, via Macchiavelli 35, 44121 Ferrara, Italy}
\email{damiano.foschi@unife.it}
\address{Hausdorff Center for Mathematics, Universit\"{a}t Bonn, Endenicher Allee 60, 53115 Bonn, Germany}
\email{dosilva@math.uni-bonn.de}
\email{thiele@math.uni-bonn.de}

\allowdisplaybreaks
\numberwithin{equation}{section}

\maketitle

\begin{abstract} In this paper we prove a sharp trilinear inequality which is motivated by a program to obtain the sharp form of the $L^2 - L^6$ Tomas-Stein adjoint restriction inequality on the circle. 
Our method uses intricate estimates for integrals of sixfold products of Bessel functions developed in a companion paper \cite{OST}. We also establish that  constants are local extremizers of the
Tomas-Stein adjoint restriction inequality as well as of another inequality appearing in the program.
\end{abstract}

\section{Introduction}
Let $(\S^1,\sigma)$ denote the unit circle in the plane equipped with its arc length measure. We are interested in the sharp version of the endpoint Tomas-Stein adjoint restriction inequality \cite{T, S} on the circle:
\begin{equation}\label{Intro_TomasStein}
\|\widehat{f\sigma}\|_{L^6(\R^2)}\leq {\bf C_{{\rm opt}}} \,\|f\|_{L^2(\S^1)},
\end{equation}
where the Fourier transform of the measure $f\sigma$ is given by 
\begin{equation*}
\widehat{f\sigma}(x)=\int_{\S^1} f(\omega)\,e^{-i x\cdot\omega}\, \d\sigma_\omega,\;\;\;(x\in\R^2),
\end{equation*}
and ${\bf C_{{\rm opt}}}$ denotes the optimal constant,
$${\bf C_{{\rm opt}}}:=\sup_{0\neq f\in L^2(\S^1)} \Phi(f);\ \ \Phi(f) := \|\widehat{f\sigma}\|_{L^6(\R^2)}\|f\|_{L^2(\S^1)}^{-1}.$$
The existence of global extremizers of $\Phi$ was recently established by Shao \cite{Sh3}.
Our first result establishes that the constant function ${\bf 1}$ is a local extremizer of $\Phi$. 
\begin{theorem}\label{Thm1}
There exists $\delta>0$ such that, whenever $\|f-{\bf 1}\|_{L^2(\mathbb{S}^1)}<\delta$, we have
$\Phi(f)\leq\Phi({\bf 1}).$
\end{theorem}
It is known that the constant function ${\bf 1}$ is a critical point of $\Phi$. Indeed, by rotational symmetry, $f={\bf 1}$ satisfies the generalized Euler-Lagrange equation $f=\lambda (|\widehat{f\sigma}|^4 \widehat{f\sigma})^{\vee}\mid_{\mathbb{S}^1}$ that characterizes critical points, see \cite[Proposition 2.1]{CQ} for details. We give the proof of Theorem \ref{Thm1} in Section \ref{sec:local_ext}. 

\smallskip

Our second and main result concerns a trilinear form related to Fourier restriction. To motivate this trilinear form, we start by using Plancherel's identity and writing
\begin{align}
\|\widehat{f\sigma}\|_{L^6(\R^2)}^6 & =(2\pi)^2 \|f\sigma\ast f\sigma\ast f\sigma\|_{L^2(\R^2)}^2 \notag \\
& =(2\pi)^2 (f\sigma) \ast (f\sigma) \ast (f\sigma) \ast (f_\star \sigma) \ast (f_\star \sigma) \ast (f_\star \sigma) (0) \notag \\
& =(2\pi)^2
\int_{(\S^1)^6} f(\omega_1) f(\omega_2)f(\omega_3)f_\star(\omega_4) f_\star(\omega_5)f_\star (\omega_6)
\,  \d\Sigma_{\vec{\omega}},\label{integralffffff}
\end{align}
where $f_\star(\omega)=\overline{f(-\omega)}$ and 
$$\d\Sigma_{\vec{\omega}}=\delta(\omega_1+\omega_2+\omega_3+\omega_4+\omega_5+\omega_6)\,\d\sigma_{\omega_1}\,\d\sigma_{\omega_2}\,\d\sigma_{\omega_3}\,\d\sigma_{\omega_4}\,\d\sigma_{\omega_5}\,\d\sigma_{\omega_6}.$$
Here $\delta$ stands for the two dimensional Dirac measure.
Note that the measure $\d\Sigma_{\vec{\omega}}$ is supported on the four dimensional manifold $\Gamma \subset (\S^1)^6$
 determined by
\begin{equation}\label{Intro_Prop5_Geom}
\omega_1+\omega_2+\omega_3+\omega_4+\omega_5+\omega_6=0.
\end{equation}

We define the trilinear form:
 \begin{equation}\label{definet}
T(h_1, h_2, h_3) := \int_{(\S^1)^6} h_1(\omega_1) h_2(\omega_2) h_3(\omega_3) \big(|\omega_4+\omega_5+\omega_6|^2-1\big)\,\d\Sigma_{\vec{\omega}}.
\end{equation}
The main result of this paper is the following monotonicity estimate, obtained in Section \ref{sec:cubic} via a spectral decomposition and a careful analysis of integrals involving Bessel functions. By antipodally symmetric function we mean a function $h$ on $\S^1$ with 
$h(\omega)={h(-\omega)}$.
\begin{theorem}\label{Thm7}
Let $h \in L^1(\S^1)$ be a nonnegative and antipodally symmetric function. Let $c = \frac{1}{2\pi} \int_{\S^1} h(\omega)  \,\d\sigma_{\omega}$ be the mean value of $h$. Then 
\begin{equation*}
T(h,h,h) \leq T(c,c,c),
\end{equation*}
with equality if and only if $h$ is constant.
\end{theorem}

This bound for the trilinear form $T$ is the penultimate step
in a six-step program that we propose to obtain the sharp form of the Tomas-Stein adjoint restriction inequality \eqref{Intro_TomasStein} and characterize its global extremizers.
 A similar program was used in \cite{F2} to obtain the sharp endpoint $L^2 - L^4$ Tomas-Stein adjoint restriction inequality on the sphere $\S^2$, and subsequently in \cite{COS} to obtain the sharp non-endpoint $L^2 - L^4$ estimate on the sphere $\S^d$ for $3\leq d \leq 6$. 
 In this paper  we complete all the steps of this program in the case of $\S^1$, except for Step 4 which remains unresolved and that we pose as a conjecture.
 
 \smallskip
 
We briefly describe each of these steps, which result in a proof of the conditional Theorem \ref{Thm2} below.

\subsubsection{Step 1. Reduction to nonnegative functions}  
Since $|f\sigma * f\sigma * f\sigma|\leq |f|\sigma * |f|\sigma * |f|\sigma$ holds pointwise, it follows that 
\begin{equation}\label{Intro_Prop3_eq1}
\big\|f\sigma * f\sigma * f\sigma\big\|_{L^2(\R^2)} \leq  \big\||f|\sigma * |f|\sigma * |f|\sigma\big\|_{L^2(\R^2)}.
\end{equation}
Here equality holds 
if and only if there is a measurable complex-valued function $h$ on the closed ball $\ov{B(3)} \subset \R^2$ of radius $3$ centered at the origin such that
\begin{equation*}
 f(\omega_1)\, f(\omega_2)\,f(\omega_3) = h(\omega_1 + \omega_2 + \omega_3) \,\big|f(\omega_1)\, f(\omega_2)\,f(\omega_3)\big|
\end{equation*}
for $\sigma^3-$a.e. $(\omega_1, \omega_2, \omega_3) \in (\S^{1})^3$. 
This can be seen as in the proof of \cite[Lemma 8]{COS}. Compare also with \cite{CS, F2}.

\subsubsection{Step 2. Reduction to antipodally symmetric functions} 

Define the nonnegative, antipodally symmetric rearrangement $f_\sharp$ of a function $f \in L^2(\S^1)$ by
$$f_\sharp:=\sqrt{\frac{|f|^2+|f_\star|^2}{2}}.$$
If $f$ is in $L^2(\S^1)$, then so is its antipodal rearrangement, with $\|f_{\sharp}\|_{L^2(\S^1)} =\|f\|_{L^2(\S^1)}$.
A simple application of the arithmetic/geometric mean inequality as in  \cite[Corollary 3.3]{F2} shows that 
\begin{equation}\label{antipodal}
\int_{(\S^1)^6} f(\omega_1) f(\omega_2)f(\omega_3)f_\star(\omega_4) f_\star(\omega_5)f_\star (\omega_6)
 \,\d\Sigma_{\vec{\omega}}
\le 
\int_{(\S^1)^6} f_\sharp(\omega_1) f_\sharp(\omega_2)f_\sharp(\omega_3)f_\sharp(\omega_4) f_\sharp(\omega_5)f_\sharp (\omega_6)
  \,\d\Sigma_{\vec{\omega}}.
\end{equation}
Here equality holds 
if and only if $f= f_{\star} = f_{\sharp}$ {\rm (}$\sigma-$a.e. in $\S^{1}${\rm)}. This follows as 
in the proof of \cite[Lemma 9]{COS}.

From inequalities~\eqref{Intro_Prop3_eq1} and~\eqref{antipodal}
it follows that
\begin{equation*}
{\bf C_{{\rm opt}}}=\sup_{0\neq f\in L^2(\S^1),\,f\ge 0,\, f=f_\star} \ \Phi(f).
\end{equation*}
We may hence assume that our candidate  $f\in L^2(\S^1)$ to being an extremizer of \eqref{Intro_TomasStein} is also a nonnegative, antipodally symmetric function.

\subsubsection{Step 3. Geometric considerations} 

Suppose that we naively try to follow the method used in \cite{F} 
 and apply the Cauchy-Schwarz inequality directly to the last integral in \eqref{integralffffff} (or in~\eqref{antipodal}).
We would obtain
\begin{align*} 
  \|\widehat{f\sigma}\|_{L^6(\R^2)}^6 &\le (2 \pi)^2
  \int_{(\S^1)^6} |f(\omega_1)|^2 |f(\omega_2)|^2 |f(\omega_3)|^2
  \, \d\Sigma_{\vec{\omega}} \\
  &= (2 \pi)^2
  \int_{(\S^1)^3} |f(\omega_1)|^2 |f(\omega_2)|^2 |f(\omega_3)|^2
  \sigma \ast \sigma \ast \sigma(\omega_1 + \omega_2 + \omega_3)
  \,\d\sigma_{\omega_1}\,\d\sigma_{\omega_2}\,\d\sigma_{\omega_3}.
\end{align*}
If the 3-fold convolution product $\sigma \ast \sigma \ast \sigma$ were a constant function
inside its support,
then the last integral would reduce to a constant multiple of $\|f\|_{L^2(\S^1)}^6$,
and we would immediately obtain  the estimate \eqref{Intro_TomasStein}.
Unfortunately, the quantity $\sigma \ast \sigma \ast \sigma(x)$
diverges logarithmically as $x$ approaches the unit circle;
the singularity of $\sigma \ast \sigma \ast \sigma$ will be described in Section \ref{sec:convolutions}.
This singularity can be neutralized if in the integral \eqref{antipodal}  
we insert an appropriate weight which vanishes when the sum of three unit vectors is again a unit vector.
This is made possible thanks to the geometrical identity illustrated in the next lemma.

\begin{lemma}
If $(\omega_1, \omega_2, \omega_3, \omega_4, \omega_5, \omega_6)\in\Gamma$, then 
\begin{equation}\label{Intro_Prop5_Geom_2}
\sum_{6\choose 3} \big(|\omega_i+\omega_j+\omega_k|^2-1\big)=16,
\end{equation}
where the sum above runs over all the ${6\choose 3}=20$ different choices of unordered distinct indices $i,j,k\in\{1,2,3,4,5,6\}$. 
\end{lemma}
For the proof, one squares \eqref{Intro_Prop5_Geom} and expands \eqref{Intro_Prop5_Geom_2} to arrive at the desired conclusion. 

Using this identity, we can write
\begin{align*}
\|\widehat{f\sigma}\|_{L^6(\R^2)}^6
&=(2\pi)^2\frac{1}{16}\sum_{6\choose 3}\int_{(\S^1)^6}f(\omega_1)f(\omega_2)f(\omega_3)f(\omega_4)f(\omega_5)f(\omega_6)\big(|\omega_i+\omega_j+\omega_k|^2-1\big)\,\d\Sigma_{\vec{\omega}}\\
& =(2\pi)^2\frac{5}{4}\int_{(\S^1)^6} f(\omega_1)f(\omega_2)f(\omega_3)f(\omega_4)f(\omega_5)f(\omega_6)\big(|\omega_4+\omega_5+\omega_6|^2-1\big)\,\d\Sigma_{\vec{\omega}}\,,
\end{align*}
since by symmetry all 20 integrals in the first line of the last display have the same numerical value.

\subsubsection{Step 4. Reduction to a trilinear problem}  At this point in the program \cite{F2}, a similar weight as $ \big(|\omega_4+\omega_5+\omega_6|^2-1\big)$ has been introduced,
albeit nonnegative. The program there continues with an application of the Cauchy-Schwarz inequality. Since our weight is partially negative,  we cannot simply apply the Cauchy-Schwarz inequality. 
Nevertheless, we pose this inequality as a conjecture:

\begin{conjecture}\label{CSconj}
Let $f\in L^2(\S^1)$ be nonnegative and antipodally symmetric. Then: 
\begin{multline}\label{Intro_CS}
\int_{(\S^1)^6} f(\omega_1)f(\omega_2)f(\omega_3)f(\omega_4)f(\omega_5)f(\omega_6)
\big(|\omega_4+\omega_5+\omega_6|^2-1\big)\,\d\Sigma_{\vec{\omega}} \\
\leq \int_{(\S^1)^6} f(\omega_1)^2f(\omega_2)^2f(\omega_3)^2\big(|\omega_4+\omega_5+\omega_6|^2-1\big)\,\d\Sigma_{\vec{\omega}}.
\end{multline}
\end{conjecture}

Numerical simulations suggest that this inequality holds. One reason to believe so is that the negative portion of the weight is small, and via
 antipodal symmetry the values of the functions on this negative portion have a strong correlation with the values of the functions on the positive part. However, the antipodal symmetry does not preserve the support of $\d\Sigma_{\vec{\omega}}$, which makes it difficult to exploit this correlation. 
 
\smallskip
 
If on the right-hand side of \eqref{Intro_CS} we replace  $\omega_4+\omega_5+\omega_6$ by $\omega_1+\omega_2+\omega_3$ and integrate out $\omega_4$, $\omega_5$ and $\omega_6$,  we obtain
an additional weight given by the
3-fold
convolution product $\sigma \ast \sigma \ast \sigma(|\omega_1+\omega_2+\omega_3|)$.
As we have already observed, this  convolution
has a logarithmic singularity at $|\omega_1+\omega_2+\omega_3|=1$, which disappears when multiplied by the weight $|\omega_1+\omega_2+\omega_3|^2-1$, in analogy
 to the program of \cite{F2}.
    
\subsubsection{Step 5. Spectral analysis of a cubic form}  
The right-hand side of \eqref{Intro_CS}  invokes the trilinear form $T$ of our main Theorem \ref{Thm7}. Thus, using \eqref{Intro_CS} and
Theorem \ref{Thm7} yields  for nonnegative, antipodally symmetric functions $f$:

\begin{equation}\label{Intro_chain}
\|\widehat{f\sigma}\|_{L^6(\R^2)}^6 \leq {(2\pi)^2}\,\frac{5}{4}\,T(f^2,f^2,f^2) 
 \leq{(2\pi)^2}\,\frac{5}{4}\,\frac{\|f\|_{L^2(\S^1)}^6}{\|{\bf 1}\|_{L^2(\S^1)}^6}\,T({\bf 1},{\bf 1},{\bf 1})
 =\frac{\|\widehat{\sigma}\|_{L^6(\R^2)}^6}{\|{\bf 1}\|_{L^2(\S^1)}^6}\|f\|_{L^2(\S^1)}^6.
\end{equation} 
This proves the first part of Theorem \ref{Thm2} below.  

\subsubsection{Step 6. Characterizing the complex-valued extremizers} If $f \in L^2(\S^{1})$ is a complex-valued extremizer of \eqref{Intro_chain}, by Theorem \ref{Thm7} we must have $|f|_{\sharp} = \gamma\,{\bf 1}$, where $\gamma >0$ is a constant. By the discussion in Step 2 above  we must have $|f| = \gamma\,{\bf 1}$. By the discussion in Step 1 above there is a measurable function $h:\ov{B(3)} \to \C$ such that 
\begin{equation*}
 f(\omega_1)\, f(\omega_2) \, f(\omega_3)= \gamma^3\, h(\omega_1 + \omega_2 + \omega_3) 
\end{equation*}
for $\sigma^3-$a.e. $(\omega_1, \omega_2, \omega_3) \in (\S^{1})^3$. We now invoke \cite[Theorem 4]{COS} (which is originally inspired in the work of Charalambides \cite{Ch}) to conclude that there exist $c \in \C\setminus \{0\}$ and $\nu \in \C^2$ such that 
$$f(\omega) = c\,e^{\nu \cdot \omega}$$
for $\sigma-$a.e. $\omega \in \S^{1}$. Since $|f|$ is constant, we must have $\Re(\nu) = 0$ and $|c| = \gamma$. This completes the proof of the following theorem.

\begin{theorem}\label{Thm2}
Assume the validity of Conjecture \ref{CSconj}. Then 
\begin{equation*}
{\bf C_{{\rm opt}}} = (2\pi)^{-1/2} \|\widehat{\sigma}\|_{L^6(\R^2)}.
\end{equation*}
Moreover, all complex-valued extremizers of \eqref{Intro_TomasStein} are given by
\begin{equation*}
f(\omega) = c \,e^{i\xi\cdot \omega},
\end{equation*}
where $c \in \C\setminus \{0\}$ and $\xi \in \R^2$.
\end{theorem}

\smallskip

The endpoint problem for the sphere $\S^2$ discussed in \cite{F2} is simpler than the above in Steps 4 and 5. In Step 4, one faces the 
 convolution of the surface measure of the sphere with itself, which has a singularity at the origin, and one can choose a nonnegative weight vanishing at the origin, so that the corresponding Step 4 follows from a plain application of the Cauchy-Schwarz inequality.
In Step 5, the analogue spectral analysis is over a bilinear rather than trilinear form. One uses the Funk-Hecke formula and properties of the Gegenbauer polynomials to show that a certain bilinear term has a sign. This is considerably simpler than the proof of Theorem \ref{Thm7}.

\smallskip

As evidence towards Conjecture \ref{CSconj} we prove the following local result  in Section \ref{sec:evidence_CS}. Define

\begin{equation}\label{Sec5_Eq1}
\Psi(f) := \int_{(\S^1)^6} \big(f(\omega_1)f(\omega_2)f(\omega_3) - f(\omega_4)f(\omega_5)f(\omega_6)\big)^2\, \big( |\omega_4 + \omega_5 + \omega_6|^2 -1 \big)\,\d\Sigma_{\vec{\omega}}.
\end{equation}
Observe that $\Psi({\bf 1})$ is identically zero
and that Conjecture \ref{CSconj} is equivalent to the fact that $\Psi(f) \geq 0$ for $f\in L^2(\S^1)$ nonnegative and antipodally symmetric.

\begin{theorem}\label{Prop12_local_CS}
There exists $\delta>0$ such that, whenever $f$ is real-valued and $\|f-{\bf 1}\|_{L^2(\mathbb{S}^1)}<\delta$, we have
$\Psi(f)\geq 0$.
\end{theorem}
Note that this result holds for all  real-valued functions, without assumption of nonnegativity nor antipodal symmetry.

\smallskip

The study of sharp Fourier restriction inequalities for the sphere $\S^d$ is quite recent, with the aforementioned works \cite{COS, CS, F2, Sh3} and the additional \cite{CS2}. The literature on sharp Fourier restriction inequalities related to the paraboloid and cone is extensive and we highlight the works \cite{BBCH, BR, C, F, HZ, K, Q}. Other interesting works on sharp Strichartz-type estimates and on the existence of extremizers for other Fourier restriction estimates include \cite{BBJP, BJ, B, FVV, FVV2, FK, HS, J, OS, OR, Ra, Sh, Sh2}.


\section{Convolutions of unit circle measures} \label{sec:convolutions}

We start by recalling a particular case of \cite[Lemma 5]{COS}.
\begin{lemma}\label{simpleconvolution}
The convolution $\sigma\ast\sigma$ is supported on the disk of radius $2$ centered at the origin, and for $|x| \le 2$ we have:
$$ \sigma \ast \sigma (x) = \frac{4}{|x| \sqrt{4 -|x|^2}}. $$
\end{lemma}
\noindent Lemma \ref{simpleconvolution} can be combined together with an additional convolution to yield
$$
\sigma \ast (\sigma \ast \sigma)(x) =
\int_{S_x} \frac{4 \d\sigma_\omega}{|x-\omega| \sqrt{4 - |x-\omega|^2}},
$$
where $S_x = \{ \omega \in \S^1: |x-\omega|\le 2 \}$.
The last integrand can be written as a function which depends only
  on the radius $r := |x|$ and on the cosine $u := \frac{x}{|x|} \cdot \omega$.
  We have that $\d\sigma_\omega = (1-u^2)^{-1/2} \du$ and,
  by applying this change of variables in the integration, we obtain the following formula.
  
\begin{lemma}
The convolution $\sigma\ast\sigma\ast\sigma$ is supported on the disk of radius $3$ centered at the origin, and for $|x| \le 3$ we have:
\begin{equation} \label{sigmasigmasigmax}
\sigma \ast \sigma \ast \sigma(x) =
\frac4r \int_{A(r)}^1 \frac{\du}{
  \sqrt{1-u^2} \sqrt{\frac{(1-r)^2}{2r} + 1-u} \sqrt{\frac{(3+r)(1-r)}{2r} + 1+u}
  },
\end{equation}
where $r = |x|$ and $A(r) := -1 + \max \{ 0, (3+r)(r-1)/(2r) \}$.
\end{lemma}

The integral \eqref{sigmasigmasigmax} diverges for $r = 1$.
Suppose $\varepsilon:= |r-1|>0$.
The contribution coming from integration over the intervals $(A(r), A(r)+\varepsilon)$ and $(1-\varepsilon^2, 1)$  remains bounded as $\varepsilon \to 0$, while the contribution coming from the integration over $[A(r)+\varepsilon, 1-\varepsilon^2]$ grows like $|\log\varepsilon|$.
We obtain, as $|x| \to 1$,

$$c\leq \frac{\sigma \ast \sigma \ast \sigma (x)}{\bigg| \log \Big| |x|-1 \Big| \bigg|}\leq C, $$
for some absolute constants $c,C>0$.


\section{Bessel functions}\label{sec:Bessel}

The main technical part of this paper uses the Bessel functions $J_n$ and estimates for integrals of sixfold products of Bessel functions that are proved in the companion paper \cite{OST}. Here we introduce the basic definitions and present the estimates from \cite{OST} in a convenient form for our purposes. 
We identify $\R^2\simeq\C$, and 
write a vector $x\in\R^2$ as a point in the complex plane $x=|x| e^{i\arg(x)}$.
For every $n \in \Z$, define 

$$e_n(x) := x^n = |x|^n e^{i n \arg(x)}.$$
Bessel functions
can be defined via
the Fourier transform of the circular harmonics. 
\begin{definition}\label{hatJn}
Let $n\in\Z$ and $x\in\R^2$. Then the Bessel function of order $n$, denoted $J_n$, is defined by
\begin{equation}\label{hatBessel}
\widehat{e_n \sigma}(x)=2\pi (-i)^n J_n(|x|)\,{|x|^{-n}e_n(x)}.
\end{equation}
\end{definition}

Bessel functions come into play via the following calculation.
We have 
 \begin{align}\label{sixfoldproduct}
(2\pi)^2 & \int_{(\S^1)^6} f_1(\omega_1) f_2(\omega_2)f_3(\omega_3)f_4(\omega_4) f_5(\omega_5)f_6 (\omega_6)
\,  \d\Sigma_{\vec{\omega}} =\int_{\R^2} \widehat{f_1\sigma}\,\widehat{f_2\sigma}\,\widehat{f_3\sigma}\,\widehat{f_4\sigma}\,\widehat{f_5\sigma}\,\widehat{f_6\sigma} \,\dx.
\end{align}
Assume that the six functions $f_j$, $1\le j\le 6$, are spherical
harmonics on $\S^1$, that is $f_j(\omega)=e_{n_j}(\omega) = \omega^{n_j}$.
Restricted to circles about the origin, the integrand on the right-hand side of \eqref{sixfoldproduct} is a spherical harmonic of index
$n:=n_1+n_2+n_3+n_4+n_5+n_6$. So unless $n=0$, the last display vanishes. If $n=0$, then the integrand is
constant on circles about the origin, and  integrating in polar coordinates yields for the last display 
 $$=(2\pi)^7 \int_0^\infty  J_{n_1}(r )J_{n_2}(r )J_{n_3}(r )J_{n_4}(r )J_{n_5}(r )J_{n_6}(r ) \,r \,\dr =:(2\pi)^7 I_{n_1,n_2,n_3,n_4,n_5,n_6}.$$
 For more general functions on $\S^1$ we write
 \begin{equation}\label{normalizationFT}
 f_j(\omega)=\sum_{n\in \Z} \widehat{f_j}(n) \, e_n(\omega)
 \end{equation}
 and obtain for \eqref{sixfoldproduct}:
 \begin{equation}\label{sixbessels}
 (2\pi)^7 \sum_{n_1+n_2+n_3+n_4+n_5+n_6=0} \widehat{f_1}(n_1)\widehat{f_2}(n_2)\widehat{f_3}(n_3)\widehat{f_4}(n_4)\widehat{f_5}(n_5)\widehat{f_6}(n_6)\,
   I_{n_1,n_2,n_3,n_4,n_5,n_6} .
\end{equation}

\smallskip

Thus we will be interested in a good understanding of the quantities  $I_{n_1,n_2,n_3,n_4,n_5,n_6}$. 
Note that the parity $J_n=J_{-n}$ for even $n$ and $J_{n}=-J_{-n}$ for odd $n$
allows us to restrict attention to these integrals for nonnegative indices. 
In particular, the following sequences (defined for $n\in\Z$) will  come into play:
\begin{equation}\label{defAlpha}
\alpha_n:= \int_0^\infty J_n^2(r) \,J_0^4(r) \,r\,\dr,
\end{equation}
\begin{equation}\label{defAlphatilde}
\widetilde{\alpha}_n:= \int_0^\infty J_n^2(r)\, J_1^2( r)\, J_0^2( r) \,r\,\dr,
\end{equation}
as well as the linear combination
\begin{align}\label{defBeta}
\beta_n &:=\int_0^{\infty} J_n^2(r) \,J_0^2(r)\, \big( 3J_1^2(r) - J_0^2(r)\big)\,r\,\dr.
\end{align}
Table \ref{table:T1} shows some of these values,
accurate to $5 \times 10^{-7}$. Computing the values of $\alpha_n$ and $\widetilde{\alpha}_n$ with Mathematica required some care which is described in the companion paper \cite[Section 8]{OST}. The values of $\beta_n$ were obtained by subtracting the values on the first column from three times the values on the second column.
\begin{table}[htb]
\begin{tabular}{ c | c | c | c }
  \hline                       
  $n$ & $\alpha_n$ & $\widetilde{\alpha}_n$ & $\beta_n$  \\
  \hline
  0 & 0.3368280 & 0.0673656 & -0.1347312 \\
  1 & 0.0673656 & 0.0423752 & 0.0597600\\
  2 &  0.0369428 & 0.0138533 & 0.0046171\\
  3 & 0.0249883 & 0.0088143 & 0.0014546\\
  4 & 0.0188523  & 0.0064847 & 0.0006018\\
  5 & 0.0151231 & 0.0051433 & 0.0003068\\
  6 & 0.0126216  & 0.0042662 & 0.0001770\\
  7 & 0.0108283 & 0.0036466 & 0.0001115\\
    8 & 0.0094804& 0.0031850 & 0.0000746  \\
  9 & 0.0084305 & 0.0028276 & 0.0000523\\
    10 & 0.0075896  & 0.0025426 & 0.0000382 \\
  \hline  
\end{tabular} 
\captionof{table}{}
\label{table:T1}
\end{table}

\noindent The companion paper \cite{OST} gives precise estimates for these sequences summarized in the following theorem.
\begin{theorem}{\rm (}cf. \cite[Theorem 1]{OST}{\rm )}\label{alpha_n}
{For $n\ge 7$ we have}
 $$\left|\alpha_n-\frac{3}{4 \pi^2 n}+\frac{3}{32 \pi^2 (n-1)n(n+1)}\right|\le \frac{1}{500 n^4};$$
$$\left|\widetilde{\alpha}_n-\frac{1}{4 \pi^2 n}-\frac{3}{32 \pi^2 (n-1)n(n+1)}\right|\le \frac{1}{500 n^4}.$$
\end{theorem}
We deduce the following estimate for the sequence $\beta_n$.
Define
\begin{equation}\label{czero}c_0=\frac 3{8\pi^2}\ .
\end{equation}
\begin{corollary}\label{beta_n}
For $n\ge 2$ even and $\varepsilon_1=0.03$, we have
 $$\left|\beta_n-\frac{c_0}{n^3}\right|<\varepsilon_1 \frac{c_0 }{n^3}.$$
\end{corollary}
\begin{proof}
For {$n\le 10$} this follows by direct checking with the values given in Table \ref{table:T1},
the tightest case being $n=2$.
 For {$n\geq 12$} one takes a linear combination of the estimates of the previous theorem to
obtain
{$$\left|\beta_n-\frac{c_0}{(n-1)n(n+1)}\right|\leq \frac{1 }{ 125 n^4}.$$
The triangle inequality then yields 
\begin{align*}
\left|\beta_n-\frac{c_0}{n^3}\right|
\leq \frac {c_0}{n^3(n^2-1)} + \frac{1 }{125 n^4} \leq \left(\frac1{143} + \frac{1}{1500c_0} \right) \frac {c_0}{n^3} < 0.025\, \frac {c_0}{n^3}.
\end{align*} }
This proves the corollary.
\end{proof}
Note that the linear combination in the corollary 
is such that the terms of order $n^{-1}$ in the asymptotics of $\alpha_n$ and $\widetilde{\alpha}_n$ cancel.
 
\smallskip

We will also need estimates for 
\begin{equation}\label{defGamma}
\gamma_{n,m}:=\int_0^\infty J_n(r) \,J_m(r) \,J_{n+m}( r) \,J_0^3(r)  \,r \,\dr,
\end{equation}
\begin{equation}\label{defGammatilde}
\widetilde{\gamma}_{n,m}:=\int_0^\infty J_n(r) \,J_m(r)\, J_{n+m}( r) \,J_1^2( r)\, J_0( r)   \,r \,\dr,
\end{equation}
and 
\begin{equation}\label{defDelta}
\delta_{n,m}:=\int_0^\infty J_n(r) \,J_m(r) \,J_{n+m}(r)\, \big(3J_1^2(r)-J_0^2(r)\big)\,J_0(r) \,r \,\dr.
\end{equation}

\noindent The values on the first two columns of Table \ref{table:T2} were again computed with Mathematica and have precision  $5 \times 10^{-8}$.
\begin{table}[htb]
\begin{tabular}{ c | c | c | c | c }
  \hline                       
  $n$ & $m$ & $\gamma_{n,m}$ & $\widetilde{\gamma}_{n,m}$ & $\delta_{n,m}$  \\
  \hline
  2 & 2  & 0.00090754 & 0.00061039 & 0.00092363 \\
  4 & 2 & 0.00019186  & 0.00012012 & 0.00016850\\
  6 & 2 & 0.00006958  & 0.00004264 & 0.00005834\\
  4 & 4 &  0.00002195  & 0.00001272 & 0.00001621\\
  6 & 4 &   0.00000498 & 0.00000281 & 0.00000345 \\
  8 & 4 &  0.00000160  & 0.00000089  & 0.00000107  \\
 10 & 4 & 0.00000064 &  0.00000035 &  0.00000041  \\
  \hline  
\end{tabular} 
\captionof{table}{}
\label{table:T2}
\end{table}

\noindent The companion paper \cite{OST} proves the following result.
\begin{theorem}{\rm (}cf. \cite[Theorem 1]{OST}{\rm )}\label{gamma_n}
For $n\ge 6$ even we have
\begin{itemize}
\item[(i)] 
$$\left|\gamma_{n,2}- \frac{15}{64 \pi^2 n(n+1)(n+2)}\right|\le \frac 1{500n^4};$$
$$\left|\widetilde{\gamma}_{n,2}- \frac{9}{64 \pi^2 n(n+1)(n+2)}\right|\le \frac 1{500n^4}.$$
\item[(ii)]  
$$\left|\gamma_{n,4}- \frac{1557}{1024 \pi^2 n(n+1)(n+2)(n+3)(n+4)}\right|\le \frac {3}{2000n^4};$$
$$\left|\widetilde{\gamma}_{n,4}- \frac{855}{1024 \pi^2 n(n+1)(n+2)(n+3)(n+4)}\right|\le \frac {3}{2000n^4}.$$
\end{itemize}

\smallskip

For $n$ and $m$ even with $n \geq m \geq 6$ we have
\begin{itemize}
\item[(iii)] 
$$\left|\gamma_{n,m}\right|\, ,\, \left|\widetilde{\gamma}_{n,m}\right|\le \frac {3}{2000n^4}.$$
 \end{itemize}
\end{theorem}
Again we obtain a simple corollary for $\delta_{n,m}$, where we recall the constant $c_0$ from \eqref{czero}.

\begin{corollary}\label{delta_n}
\begin{itemize}
\item[(i)]  For $n\ge 2 $ even and $\varepsilon_2=0.11$ we have 
$$|\delta_{n,2}|\le (1+\varepsilon_2) \frac{c_0}{2 n^{3/2} (n+2)^{3/2}}.$$
 \item[(ii)]  For $n\ge 4$ even and $\gamma_3=1.3$ we have
$$\left|\delta_{n,4}- \frac{21c_0}{8 n(n+1)(n+2)(n+3)(n+4)}\right|\le \gamma_3 \frac {c_0}{8n^4}.$$
\item[(iii)] For $n$ and $m$ even with $n \geq m \geq 6$ and again $\gamma_3=1.3$ we have
$$|\delta_{n,m}| \le \gamma_3 \frac {c_0}{8 n^4}.$$
\end{itemize}
\end{corollary}
\begin{proof}
We begin with inequality (i). For $n=2,4,6$ this is verified directly with Table \ref{table:T2}.
Again the tightest case is $n=2$. For $n\ge 8$, from Theorem \ref{gamma_n} we have
\begin{align*}
|\delta_{n,2}|\le \frac{c_0}{2 n(n+1)(n+2)}+\frac 1{125 n^4} \le \frac{c_0}{2 n^{3/2} (n+2)^{3/2}}+\left(\frac {10} 8\right)^{3/2} \!\frac 1{1000 \,n^{3/2}(n+2)^{3/2} }\, , 
\end{align*}
which is less than the desired quantity. Inequalities (ii) and  (iii) follow from Theorem \ref{gamma_n} via the estimate
{$$\frac{3}{500 }\le 1.3\, \frac{3}{64 \pi^2} .$$}
This completes the proof of the corollary.
\end{proof}


\section{Proof of Theorem \ref{Thm1}: Constants are local extremizers of the extension inequality }\label{sec:local_ext}

In this section we follow the outline of \cite[Section 16]{CS} to prove Theorem \ref{Thm1}. 
Note that 
\begin{itemize}
\item[(i)] {$\Phi(f)=\Phi(\lambda f)$} for all $\lambda>0$;
\item[(ii)] $\Phi(f)\leq \Phi(|f|) \leq \Phi(|f|_{\sharp})$;
\item[(iii)] $\||f|_{\sharp} - {\bf 1}\|_{L^2(\mathbb{S}^1)}\leq \||f| - {\bf 1}\|_{L^2(\mathbb{S}^1)} \leq \|f - {\bf 1}\|_{L^2(\mathbb{S}^1)}$.
\end{itemize}
We may therefore restrict attention to functions of the form 
$$f={\bf 1}+\varepsilon g,$$ 
where $0\leq\varepsilon\leq\delta$, $g\perp {\bf 1}$, $\|g\|_{L^2(\mathbb{S}^1)}=1$, with $g$ real-valued and antipodally symmetric. A straightforward calculation gives the Taylor expansion
\begin{equation}\label{Sec3_eq1}
\Phi(f)^6 = \Phi({\bf 1})^6 + (2\pi \varepsilon)^2 \|{\bf 1}\|_{2}^{-6}
\big( 15 g\sigma \ast g\sigma \ast \sigma \ast \sigma \ast \sigma \ast \sigma(0) -3 \sigma \ast \sigma \ast \sigma \ast \sigma \ast \sigma \ast \sigma(0) \, \|{\bf 1}\|_2^{-2} \, \|g\|_2^2 \big) + O(\varepsilon^3),
\end{equation}
where $O(\varepsilon^3)$ denotes a quantity whose absolute value is majorized by $C \varepsilon^3$, uniformly for $g$ satisfying $\|g\|_{L^2(\S^1)}\leq 1$. Note that we do not have a term in $\varepsilon$ since 
$$g\sigma \ast \sigma \ast \sigma \ast \sigma \ast \sigma\ast\sigma(0)=0$$ 
due to the discussion after
\eqref{sixfoldproduct}
and the fact that $g\perp {\bf 1}$, i.e. $\widehat{g}(0) = 0$. From \eqref{Sec3_eq1} it suffices to show that
\begin{equation*}
5 \sup_{\|g\|_2 = 1}g\sigma \ast g\sigma \ast \sigma \ast \sigma \ast \sigma\ast\sigma(0)<  \sigma \ast \sigma \ast \sigma \ast \sigma\ast\sigma\ast\sigma(0)  \,\|{\bf 1}\|_2^{-2} \|g\|_2^2.
\end{equation*}
Using \eqref{sixbessels} together with the fact that $g$ is real with mean zero and antipodally symmetric, and therefore can only have even nonzero Fourier coefficients, this reduces to\footnote{Throughout this paper, we let $(2\Z)^\times:=2\Z\setminus\{0\}$ and $\Z^\times:=\Z\setminus\{0\}$. Similarly for $(2\N)^\times$, where $\N:=\{0,1,2,\ldots\}$.}
\begin{equation}\label{crux2}
5 \sum_{n\in(2\Z)^\times} |\widehat{g}(n)|^2 \alpha_n< \sum_{n\in(2\Z)^\times} |\widehat{g}(n)|^2 \alpha_0,
\end{equation}
{where we have used the fact that
$\|g\|_{L^2(\mathbb{S}^1)}^2 = 2\pi \sum_{n\in(2\Z)^\times} |\widehat{g}(n)|^2$}. Estimate \eqref{crux2} will follow from $5\alpha_n<\alpha_0$ for all $n\in (2\Z)^\times$.
This in turn follows from Theorem \ref{alpha_n} and  Table \ref{table:T1}.\footnote{However, it can be shown using integration by parts that $5\alpha_1=\alpha_0$.} In particular, for $n\ge 10$ we conclude from Theorem \ref{alpha_n} that
  $$\alpha_n\le \frac 3{4\pi^2 n}+\frac{3}{32 \pi^2 (n-1)n(n+1)}+\frac{1}{500 n^4}\le \frac 1{50}.$$ 
This completes the proof of Theorem \ref{Thm1}.

\smallskip

\noindent {\sc Remark:} By using Theorem \ref{alpha_n}, we appeal to the companion paper \cite{OST}. However, this particular consequence \eqref{crux2} is a very simple case of the analysis in \cite{OST}, and for self containment we sketch a proof of the bound $5\alpha_n<\alpha_0$ for all $n\in (2\Z)^\times$. One first reduces the estimate to an estimate for integrals over bounded domains, that is to
\begin{equation}\label{seven}
7\int_0^{100} J_n^2( r) \,J_0^4(r )  \,r \,\dr  < \int_0^{100} J_0^6(r ) \,r \, \dr,
\end{equation}
by establishing bounds for the tails, that is
$$25\int_{100}^\infty  J_n^2( r)\, J_0^4(r )\, r\, \dr, \  200 \int_{100}^\infty  J_0^6( r) \,r\, \dr 
 < \int_0^{100} J_0^6(r )\, r \, \dr\,.$$
 To see these tail bounds,  one estimates
 the left-hand sides using the well known bounds
$$\left|J_0( r)-\left(\frac 2{\pi r} \right)^{1/2} \cos\left(r-\frac{\pi}{4}\right)\right|\le   r^{-3/2} $$
and 
 $$|J_n(r )|\le r^{-1/3}\ ,$$
for all $n\ge 0$. A sharper form of the latter inequality can be found in \cite{L}, while the former is reviewed in \cite{OST}.
The right-hand sides are then evaluated numerically. Here, we assume to have a sufficiently accurate evaluation of Bessel functions
at hand such as, for example, provided by the Mathematica package. Moreover, Riemann sums with step size $1000^{-1}$ will give sufficient accuracy.  To see the estimate \eqref{seven} for the integrals over bounded domains, 
in case $n\le 200$ one simply evaluates likewise numerically.  To see the estimate for $n>200$, one estimates the 
left-hand side using $|J_0|\le 1$ and the well-known estimate 
 $$J_n(r )\le \frac{r^n}{2^n n!} $$
for all $n\ge 0$ and $r> 0$, reviewed in \cite{OST}. This completes the outline of the proof that $5\alpha_n<\alpha_0$ for $n\in (2\Z)^\times$. As a final remark, note that  a more refined analysis would allow to reduce the numerical component of the proof.


\section{Proof of Theorem \ref{Thm7}: The sharp trilinear inequality}\label{sec:cubic}

We shall prove Theorem \ref{Thm7} for $h$ being a nonnegative and antipodally symmetric trigonometric polynomial. The result for a general $h \in L^1(\S^1)$ nonnegative and antipodally symmetric follows by a standard approximation argument, for example by convolving with the F\'{e}jer kernel, since the map $h \mapsto T(h,h,h)$ is continuous on $L^1(\S^1)$. To pass the case of equality to the limit in the approximation argument, we observe from the proof below that each nonzero even Fourier coefficient of $h$ has a strictly negative contribution. 

\smallskip

Let $h$ be a nonnegative and antipodally symmetric trigonometric polynomial. Write 
$$h = c + g,$$ 
with $g \perp {\bf 1}$ and $c = \frac{1}{2\pi} \int_{\S^1} h(\omega)  \,\d\sigma_{\omega}$. By the assumptions on $h$, we have that $\widehat{h}(-n)=\ov{\widehat{h}(n)} $ for every $n\in\Z$, and that $\widehat{h}(n)\neq 0$ only if $n\in2\Z$. The analogous statements hold for $g$, and moreover $\widehat{g}(0)=0$. By linearity and symmetry, one can immediately check that

\begin{equation*}
T(h,h,h) = T(c,c,c) + 3 T(c,c,g) + 3 T(c,g,g) + T(g,g,g).
\end{equation*}
The strategy to prove Theorem \ref{Thm7} will be to analyze each of these summands separately. It turns out that the linear term is zero, the bilinear term is nonpositive, and the trilinear term can be controlled in absolute value by the bilinear term. Once we establish these facts, which are the subject of the remainder of this section, the result follows.

\subsection{Linear term} 
Let $R_\theta\omega$ denote the rotation of $\omega$ by the angle $\theta$ counterclockwise around the origin.
Denote $R_\theta g(\omega)=g(R_{\theta}\omega) $. Then it is immediate from the definition that 
$T(R_\theta f_1, R_\theta f_2,R_\theta f_3)=T(f_1,f_2,f_3)$ for any functions $f_1,f_2,f_3$ in $L^2(\S^1)$. For the linear term of
our expansion this means that
$$T(c,c,f)=T(c,c,R_\theta f).$$
Hence $f\mapsto T(c,c,f)$ is a rotation invariant linear functional on $L^2(\S^1)$, and therefore it is a multiple of the averaging operator.
 Since $g$ has mean zero, we obtain $T(c,c,g)=0$.

\subsection{Bilinear term} 
We expand  
\begin{equation}\label{expandsquare}
|\omega_4+\omega_5+\omega_6|^2-1=2\,(1+\omega_4\cdot\omega_5+\omega_5\cdot\omega_6+\omega_6\cdot\omega_4).
\end{equation}
Thus the integral \eqref{definet} defining $T(c,g,g)$ splits into a sum of four terms, the last three of which are identical by symmetry considerations.
 We first consider 
$$I:=\int_{(\S^1)^6} g(\omega_2)g(\omega_3) \, \d\Sigma_{\vec{\omega}}.$$
It follows by calculations as the ones leading to \eqref{sixbessels} that  
\begin{align} \label{I}
\begin{split}
I& = g\sigma * g\sigma *\sigma *\sigma *\sigma *\sigma (0)\\
& =(2\pi)^{-2}\sum_{n\in(2\Z)^\times}\sum_{m\in (2\Z)^\times}  \widehat{g}(n)  \widehat{g}(m)  \int_{\R^2}
 \widehat{{e_n\sigma}}\,\widehat{{e_m\sigma}}\,
 \widehat{\sigma}\,\widehat{\sigma}\,\widehat{\sigma}\,\widehat{\sigma}\,\dx\\
 & =(2\pi)^{-2}\sum_{n\in(2\Z)^\times}  |\widehat{g}(n)|^2  \int_{\R^2}
 \widehat{{e_n\sigma}}\,\widehat{{e_{-n}\sigma}}\,
 \widehat{\sigma}\,\widehat{\sigma}\,\widehat{\sigma}\,\widehat{\sigma}\,\dx\\
& =(2\pi)^{5}\sum_{n\in(2\Z)^\times} |\widehat{g}(n)|^2 \, \alpha_n,
\end{split}
\end{align}
where the sequence $\{\alpha_n\}$ was defined in \eqref{defAlpha}. 

\smallskip

We now focus on the second integral,
$$II:=\int_{(\S^1)^6} g(\omega_2)g(\omega_3) (\omega_4\cdot\omega_5)\,  \d\Sigma_{\vec{\omega}}.$$
Observe that, using the algebra of complex numbers, we can write
 $${\omega_4\cdot \omega_5=\cos(\arg(\omega_4)-\arg(\omega_5) )=
\Re(\omega_4 \ov{\omega_5}) =
\frac 12 \big(\omega_4 \ov{\omega_5} + \ov{\omega_4} \omega_5 \big) =
\frac 12 \big(e_1(\omega_4) e_{-1}(\omega_5) + e_{-1}(\omega_4) e_1(\omega_5) \big)}.$$
By symmetry we obtain 
$$II=\int_{(\S^1)^6} g(\omega_2)g(\omega_3) \, {e_1}(\omega_4) {e_{-1}}(\omega_5) \, \d\Sigma_{\vec{\omega}}.$$
By a similar calculation as for the first integral we obtain
\begin{align}\label{II}
\begin{split}
II & =(2\pi)^{-2}\int_{\R^2} \widehat{g\sigma}\,\widehat{g\sigma}\,
\widehat{{e_1\sigma}}\,\widehat{{e_{-1}\sigma}}\,\widehat{\sigma}\,\widehat{\sigma}\,\dx\\
 & =(2\pi)^{-2}\sum_{n\in(2\Z)^\times}\sum_{m\in (2\Z)^\times}  \widehat{g}(n)  \widehat{g}(m)  \int_{\R^2}
 \widehat{{e_n\sigma}}\,\widehat{{e_m\sigma}}\,
 \widehat{{e_1\sigma}}\,\widehat{{e_{-1}\sigma}}\,\widehat{\sigma}\,\widehat{\sigma}\,\dx\\
 & =-(2\pi)^{5}\sum_{n\in(2\Z)^\times} |\widehat{g}(n)|^2 \, \widetilde{\alpha}_n,
\end{split}
\end{align}
where the sequence $\{\widetilde{\alpha}_n\}$ was defined in \eqref{defAlphatilde}. Finally we obtain
$$T(c,g,g)=2c(I+3II)= -2c (2\pi)^{5}\sum_{n\in(2\Z)^\times} |\widehat{g}(n)|^2 \, \beta_n,$$
with $\{\beta_n\}$ as defined in \eqref{defBeta}. Since the numbers $\beta_n$ are positive by Corollary \ref{beta_n}, this establishes that the bilinear term $T(c,g,g)$ is nonpositive.

\subsection{Trilinear term} Identity \eqref{expandsquare} allows us to again express $T(g,g,g)$ as a sum of four integrals, the last three of which are identical by symmetry considerations. We start by computing the first one similarly
to the previous calculations:

\begin{align*}
I & =\int_{(\S^1)^6} g(\omega_1) g(\omega_2) g(\omega_3) \, \d\Sigma_{\vec{\omega}} 
 =(2\pi)^{-2}\int_{\R^2} \widehat{g\sigma}\,\widehat{g\sigma}\,
\widehat{g\sigma}\,\widehat{\sigma}\,\widehat{\sigma}\,\widehat{\sigma}\,\dx\\
 & =(2\pi)^{-2}\sum_{n\in(2\Z)^\times}\sum_{m\in (2\Z)^\times} \sum_{k\in (2\Z)^\times}  
\widehat{g}(n)  \widehat{g}(m)  \widehat{g}(k)  \int_{\R^2}
 \widehat{{e_n \sigma}}\,\widehat{{e_m \sigma}}\,
 \widehat{{e_k \sigma}}\,\widehat{\sigma}\,\widehat{\sigma}\,\widehat{\sigma}\,\dx\\
& =(2\pi)^{5}\sum_{n\in(2\Z)^\times} \sum_{m\in (2\Z)^\times} \widehat{g}(n)\widehat{g}(m)
 \ov{\widehat{g}(n+m)} \gamma_{n,m},
\end{align*}
with $\{\gamma_{n,m}\}$ as defined in \eqref{defGamma}. For the second integral we obtain similarly 
\begin{align*}
II& =\int_{(\S^1)^6} g(\omega_1) g(\omega_2) g(\omega_3) (\omega_4 \cdot \omega_5)
\, \d\Sigma_{\vec{\omega}}\\
& =-(2\pi)^{5}\sum_{n\in(2\Z)^\times} \sum_{m\in (2\Z)^\times} \widehat{g}(n)\widehat{g}(m)
 \ov{\widehat{g}(n+m)} \, \widetilde{\gamma}_{n,m},
\end{align*}
with $\{\widetilde{\gamma}_{n,m}\}$ as defined in \eqref{defGammatilde}. Summarizing, we obtain
$$T(g,g,g)=2(I+3II)=
-2(2\pi)^{5}\sum_{n\in(2\Z)^\times} \sum_{m\in (2\Z)^\times} \widehat{g}(n)\widehat{g}(m)
 \ov{\widehat{g}(n+m)} \, \delta_{n,m},$$
with $\{\delta_{n,m}\}$ as defined in \eqref{defDelta}.  

\subsection{Bilinear controls trilinear} We want to show that the trilinear term we just computed is controlled in absolute value by the bilinear term $-3T(c,g,g)$. Since $h\geq 0$, the constant $c$ is given by (recall that we are using the normalization \eqref{normalizationFT} for the Fourier series)
$$c=\frac{\|h\|_1}{2\pi}=\widehat{h}(0)=\|\widehat{h}\|_\infty.$$
Observe that $\widehat{g}(n)=\widehat{h}(n)$ for $n\neq 0$. Our task can thus be reformulated as the following statement:
\begin{equation*}
\left|\sum_{n,m, n+m \in (2\Z)^\times}\widehat{h}(n) \widehat{h}(m)\ov{\widehat{h}(n+m)}\delta_{n,m}\right|
\leq {3\|\widehat{h}\|_\infty}\sum_{n\in (2\Z)^\times} |\widehat{h}(n)|^2\beta_{n}.
\end{equation*}
Letting $k = -m-n$, we further simplify the problem by using the symmetries of the planar lattice $\big((2\Z)^\times\big)^3 \cap \{n+m+k=0\}$. We have two possibilities: (i) two numbers positive and one negative or (ii) two numbers negative and one positive.
Since $\widehat{h}(n) = \ov{\widehat{h}(-n)}$ for every $n\in\Z$, the two cases are actually the same, and so we work with case (i) only.
 In this case, we consider the instances where $k$ is negative. By the triangle inequality, it suffices   to show that
\begin{equation}\label{task2finalfinal}
\left|\sum_{n,m\in(2\N)^\times}\widehat{h}(n) \widehat{h}(m)\ov{\widehat{h}(n+m)}\delta_{n,m}\right|
\leq {\|\widehat{h}\|_\infty}\sum_{n\in (2\N)^\times} |\widehat{h}(n)|^2\beta_{n}.
\end{equation} 
 
Recall that $c_0 = 3/8\pi^2$ and define
$$\eta_{n,4} = \frac{21c_0}{8}\frac{1}{n(n+1)(n+2)(n+3)(n+4)}$$
and 
$$\widetilde{\delta}_{n,4} =  \delta_{n,4} - \eta_{n,4} .$$
 For $n\in\{6,8,\ldots\}$ and $m\in\{6,\ldots, n\}$, define
$$\widetilde{\delta}_{n,m} =  \delta_{n,m}.$$

We break the left-hand side of \eqref{task2finalfinal} into 6 sums. The first two are the terms for which $\min(n,m)=2$, sorted into those for which $n\le m$ and those for which $n>m$. The next two are the terms for which $\min(n,m)=4$, in which we have isolated the main contribution $\eta_{n,4}$. The last two sums are the terms with $\min(n,m)\geq 4$ with the residual contribution $\widetilde{\delta}_{n,m}$. 
\begin{align*}
(LHS) & \leq  \left|\sum_{\substack{n\in2\N: \\ 2\leq n}}\widehat{h}(n) \widehat{h}(2)\ov{\widehat{h}(n+2)}\delta_{n,2}\right| +  \left|\sum_{\substack{n\in2\N: \\ 2 < n}}\widehat{h}(n) \widehat{h}(2)\ov{\widehat{h}(n+2)}\delta_{n,2}\right|\\
& \ \ \ +   \left|\sum_{\substack{n\in2\N: \\ 4\leq n}}\widehat{h}(n) \widehat{h}(4)\ov{\widehat{h}(n+4)}\eta_{n,4}\right| +  \left|\sum_{\substack{n\in2\N: \\ 4 < n}}\widehat{h}(n) \widehat{h}(4)\ov{\widehat{h}(n+4)}\eta_{n,4}\right|\\
&  \ \ \ \ \ \ +  \left|\sum_{\substack{n,m\in2\N: \\ 4\leq m \leq n}}\widehat{h}(n) \widehat{h}(m)\ov{\widehat{h}(n+m)}\widetilde{\delta}_{n,m}\right| + \left|\sum_{\substack{n,m\in2\N: \\ 4\leq m < n}}\widehat{h}(n) \widehat{h}(m)\ov{\widehat{h}(n+m)}\widetilde{\delta}_{n,m}\right|\\
& = S_1 + S_2 + S_3 + S_4 + S_5 + S_6.
\end{align*}

\subsubsection{Analysis of $S_1$} We treat these terms in a special way so as to not have to estimate $\widehat{h}(2)$ by ${\|\widehat{h}\|_\infty}$ as in $S_5$ and $S_6$.  
Using Corollary \ref{delta_n} and the Cauchy-Schwarz inequality, we proceed as follows:
\begin{align*}
S_1 &\leq |\widehat{h}(2)|(1 + \varepsilon_2) \frac{c_0}{2}\sum_{\substack{n\in(2\N)^\times}}\frac{|\widehat{h}(n)|}{n^{3/2}} \frac{ |\widehat{h}(n+2)|}{(n+2)^{3/2}}\\
& \leq |\widehat{h}(2)|(1 + \varepsilon_2) \frac{c_0}{2} \left( \sum_{\substack{n\in(2\N)^\times}}\frac{|\widehat{h}(n)|^2}{n^{3}}\right)^{1/2}  \left( \sum_{\substack{n\in2\N: \\ 4\leq n}}\frac{|\widehat{h}(n)|^2}{n^{3}}\right)^{1/2}  .
\end{align*}  
Let $|\widehat{h}(2)| = x$ and $\sum_{\substack{n\in(2\N)^\times}}\frac{|\widehat{h}(n)|^2}{n^{3}} = S$. We seek to maximize 
$$x \mapsto \left[x^2\left (S - \frac{x^2}{8}\right)\right]^{1/2}.$$
This maximum occurs when $x^2 = 4S$. We also note that 
\begin{equation}\label{S-zeta}
S \leq \frac{c^2}{8}\zeta(3),
\end{equation} 
where $\zeta(s) = \sum_{n=1}^{\infty} \frac{1}{n^s}$ is the Riemann zeta-function. At the point of maximum we then have that
$$\left[x^2\left (S - \frac{x^2}{8}\right)\right]^{1/2} = \left[ 2S^2\right]^{1/2} \leq \frac{\sqrt{\zeta(3)}}{2}\, c\, S^{1/2} < (0.55)\, c\, S^{1/2}.$$
Hence 
\begin{align*}
S_1 < (1 + \varepsilon_2) \,c_0\,(0.275)\, c\, S.
\end{align*}
Using Corollary \ref{beta_n} we then arrive at
\begin{align}\label{bound_final_S_1}
S_1 & < \left[ \frac{(1 + \varepsilon_2) (0.275)}{(1 - \varepsilon_1)}\right] \,c\, \sum_{\substack{n\in(2\N)^\times}} |\widehat{h}(n)|^2\,\beta_n.
\end{align}

\subsubsection{Analysis of $S_2$} We follow the same outline as above, and now we obtain a slight improvement due to the restricted summation indices. In fact,
\begin{align*}
S_2 &\leq |\widehat{h}(2)|(1 + \varepsilon_2) \frac{c_0}{2}\sum_{\substack{n\in2\N: \\ 2 < n}}\frac{|\widehat{h}(n)|}{n^{3/2}} \frac{ |\widehat{h}(n+2)|}{(n+2)^{3/2}}\\
& \leq |\widehat{h}(2)|(1 + \varepsilon_2) \frac{c_0}{2}  \sum_{\substack{n\in2\N: \\ 4\leq n}}\frac{|\widehat{h}(n)|^2}{n^{3}}.
\end{align*}  
Again we let $|\widehat{h}(2)| = x$ and $\sum_{\substack{n\in(2\N)^\times}}\frac{|\widehat{h}(n)|^2}{n^{3}} = S$. We now seek to maximize
$$x \mapsto x\left (S - \frac{x^2}{8}\right).$$
The maximum occurs when $x = \sqrt{8S/3}$. Using \eqref{S-zeta}, at the point of maximum we have that
$$x\left (S - \frac{x^2}{8}\right) = \sqrt{\frac{8S}{3}}\frac{2S}{3} \leq \frac{2\sqrt{\zeta(3)}}{3\sqrt{3}} \,c \, S < (0.422) c\, S.$$
Using Corollary \ref{beta_n}, this leads to 
\begin{align}\label{bound_final_S_2}
S_2 < \left[ \frac{(1 + \varepsilon_2) \, (0.211)}{(1 - \varepsilon_1)}\right] \,c\, \sum_{\substack{n\in(2\N)^\times}} |\widehat{h}(n)|^2\,\beta_n.
\end{align}

\subsubsection {Analysis of $S_3$} \label{subsec_analysis_S_3} First notice that
\begin{align*}
S_3 &\leq |\widehat{h}(4)|\,\frac{21c_0}{8}\sum_{\substack{n\in2\N: \\ 4\leq n}}\frac{|\widehat{h}(n)|}{n^{3/2}} \frac{ |\widehat{h}(n+4)|}{(n+4)^{3/2}}\, \frac{n^{3/2}(n+4)^{3/2}}{n(n+1)(n+2)(n+3)(n+4)}.
\end{align*}
Note that the function 
$$x \mapsto \frac{x^{3/2}(x+4)^{3/2}}{x(x+1)(x+2)(x+3)(x+4)}$$
is decreasing on $[4,\infty)$. Therefore 
\begin{align*}
S_3 & \leq  |\widehat{h}(4)|\,\frac{21c_0}{8}\, \frac{4\sqrt{2}}{5\times6\times7}\sum_{\substack{n\in2\N: \\ 4\leq n}}\frac{|\widehat{h}(n)|}{n^{3/2}} \frac{ |\widehat{h}(n+4)|}{(n+4)^{3/2}}.
\end{align*}
Using the Cauchy-Schwarz inequality, we then obtain that
\begin{align*}
S_3 & \leq  |\widehat{h}(4)|\,\frac{\sqrt{2}c_0}{20} \left( \sum_{\substack{n\in2\N: \\ 4\leq n}}\frac{|\widehat{h}(n)|^2}{n^{3}}\right)^{1/2}  \left( \sum_{\substack{n\in2\N: \\ 6\leq n}}\frac{|\widehat{h}(n)|^2}{n^{3}}\right)^{1/2} .
\end{align*}  
Now let $|\widehat{h}(4)| = x$ and $\sum_{\substack{n\in2\N: \\ 4\leq n}}\frac{|\widehat{h}(n)|^2}{n^{3}} = T$. We want to maximize 
$$x \mapsto \left[x^2\left (T - \frac{x^2}{64}\right)\right]^{1/2}.$$
This maximum occurs when $x^2 = 32T$. Note also that 
\begin{equation}\label{T-zeta}
T \leq \frac{c^2}{8}(\zeta(3) -1).
\end{equation}
At the point of maximum, we then have that
$$\left[x^2\left (T - \frac{x^2}{64}\right)\right]^{1/2} = 4T \leq \frac{4 \sqrt{\zeta(3) -1}}{\sqrt{8}}\, c\, T^{1/2}. $$
Hence 
\begin{align*}
S_3 \leq  c_0 \frac{\sqrt{\zeta(3)-1}}{10}\, c\, T <  c_0 \,(0.045)\, c\, T, 
\end{align*}
and from Corollary \ref{beta_n} we arrive at
\begin{align}\label{bound_final_S_3}
S_3 & < \frac{(0.045)}{(1 - \varepsilon_1)} \,c\sum_{\substack{n\in2\N: \\ 4\leq n}} |\widehat{h}(n)|^2\,\beta_n.
\end{align}

\subsubsection{Analysis of $S_4$} We follow the same outline as in the analysis of $S_3$ to get
\begin{align*}
S_4 & \leq  |\widehat{h}(4)|\frac{\sqrt{2}c_0}{20}  \sum_{\substack{n\in2\N: \\ 6\leq n}}\frac{|\widehat{h}(n)|^2}{n^{3}}.
\end{align*}  
Again we let $|\widehat{h}(4)| = x$ and $\sum_{\substack{n\in2\N: \\ 4\leq n}}\frac{|\widehat{h}(n)|^2}{n^{3}} = T$. We now seek to maximize
$$x \mapsto x\left (T - \frac{x^2}{64}\right).$$
The maximum occurs when $x = \sqrt{64T/3}$. Using \eqref{T-zeta}, at the point of maximum we have that
$$x\left (T - \frac{x^2}{64}\right) = \sqrt{\frac{64T}{3}}\frac{2T}{3} \leq \frac{8}{\sqrt{3}}\, \frac{2}{3} \,\frac{\sqrt{\zeta(3)-1}}{\sqrt{8}}\,c
\, T.$$
Hence
\begin{equation*}
S_4 \leq c_0\,\frac{\sqrt{2}}{20} \,\frac{8}{\sqrt{3}}\, \frac{2}{3} \,\frac{\sqrt{\zeta(3)-1}}{\sqrt{8}} \,c\,T < c_0 \,(0.035)\,c\,T,
\end{equation*}
and from Corollary \ref{beta_n} we arrive at 
\begin{align}\label{bound_final_S_4}
S_4 < \frac{ (0.035)}{(1 - \varepsilon_1)} \,c  \sum_{\substack{n\in2\N: \\ 4\leq n}} |\widehat{h}(n)|^2\,\beta_n.
\end{align}

\subsubsection {Analysis of $S_5$} From Corollary \ref{beta_n} and Corollary \ref{delta_n}, for every positive even integers $m$ and $n$ satisfying $4\leq m\leq n$, we have that
\begin{equation}\label{An_S_5_eq1}
\frac{|\widetilde{\delta}_{n,m}|}{\beta_n^{1/2}\beta_m^{1/2}}\leq \gamma_3\,\frac{c_0}{8n^4} \frac{n^{3/2}\, m^{3/2}}{(1- \varepsilon_1)c_0} \leq \frac{\gamma_3}{8(1- \varepsilon_1)n}. 
\end{equation}
Using \eqref{An_S_5_eq1}, it follows that  
\begin{align*}
S_5 &\leq \|\widehat{h}\|_\infty\sum_{\substack{n,m\in2\N: \\ 4\leq m\leq n}}|\widehat{h}(n)|\,\beta_n^{1/2} \,|\widehat{h}(m)|\,\beta_m^{1/2}\,\frac{|\widetilde{\delta}_{n,m}|}{\beta_n^{1/2}\beta_m^{1/2}}\\
&\leq \frac{\gamma_3}{8(1- \varepsilon_1)} \|\widehat{h}\|_\infty \sum_{\substack{n\in2\N: \\ 4\leq n}}|\widehat{h}(n)|\,\beta_n^{1/2} \left(\frac{\sum_{\substack{m\in2\N: \\ 4\leq m \leq n} }|\widehat{h}(m)|\,\beta_m^{1/2}}{n} \right)\\
&\leq \frac{\gamma_3}{16(1- \varepsilon_1)} \|\widehat{h}\|_\infty \sum_{\substack{n\in2\N: \\ 4\leq n}}|\widehat{h}(n)|\,\beta_n^{1/2} \left(\frac{\sum_{\substack{m\in2\N: \\ 4\leq m \leq n} }|\widehat{h}(m)|\,\beta_m^{1/2}}{(n/2)-1} \right).
\end{align*}
This last term can be estimated using the Cauchy-Schwarz inequality yielding
\begin{align}\label{An_S_5_eq2}
S_5 \leq \frac{\gamma_3}{16(1- \varepsilon_1)} \|\widehat{h}\|_\infty \left(\sum_{\substack{n\in2\N: \\ 4\leq n}}|\widehat{h}(n)|^2\beta_n\right)^{1/2}  \left(\sum_{\substack{n\in2\N: \\ 4\leq n}}\left(\frac{\sum_{\substack{m\in2\N: \\ 4\leq m \leq n} }|\widehat{h}(m)|\,\beta_m^{1/2}}{(n/2)-1} \right)^2\right)^{1/2}.
\end{align}
We now recall a sharp version of Hardy's inequality for sequences.

\begin{lemma}{\rm (}Hardy's inequality, cf. \cite[p.~239]{HLP}{\rm )}  
Given any sequence $\{a_n\}$ of nonnegative real numbers, we have
$$\sum_{n=1}^\infty\left(\frac{a_1+a_2+\ldots+a_n}{n}\right)^2\leq 4\sum_{n=1}^\infty a_n^2.$$
\end{lemma}

Using Hardy's inequality in \eqref{An_S_5_eq2}, with $a_{j-1}=|\widehat{h}(2j)|\,\beta_{2j}^{1/2}$, for $2 \leq j \leq \frac{n}{2}$, yields
\begin{align}\label{bound_final_S_5}
S_5  \leq \frac{\gamma_3}{8(1- \varepsilon_1)} \|\widehat{h}\|_\infty \sum_{\substack{n\in2\N: \\ 4\leq n}}|\widehat{h}(n)|^2\beta_n.  
\end{align}

\subsubsection{Analysis of $S_6$} For $S_6$ we have (at least) the same bound \eqref{bound_final_S_5} as for $S_5$. This is sufficient for our purposes.

\subsubsection{Conclusion} Putting together the estimates \eqref{bound_final_S_1}, \eqref{bound_final_S_2}, \eqref{bound_final_S_3}, \eqref{bound_final_S_4} and \eqref{bound_final_S_5} (twice), and recalling  $c=\|\widehat{h}\|_\infty$, we conclude that
\begin{align*}
S_1 + S_2 + S_3 + S_4 + S_5 + S_6 < \left[ \frac{0.08 + (1 + \varepsilon_2) (0.486) + \frac{\gamma_3}{4}}{(1 - \varepsilon_1)}\right] \|\widehat{h}\|_\infty \sum_{\substack{n\in2\N: \\ 2\leq n}}|\widehat{h}(n)|^2\beta_n.
\end{align*}
The values of $\varepsilon_1 = 0.03$, $\varepsilon_2 = 0.11$ and $\gamma_3 = 1.3$ provided by Corollaries \ref{beta_n} and \ref{delta_n} guarantee that
$$ \left[ \frac{0.08 + (1 + \varepsilon_2) (0.486) + \frac{\gamma_3}{4}}{(1 - \varepsilon_1)}\right]  < 0.974 < 1.$$
This establishes \eqref{task2finalfinal} and concludes the proof of Theorem \ref{Thm7}.


\section{Proof of Theorem \ref{Prop12_local_CS}: A local estimate of Cauchy-Schwarz type} \label{sec:evidence_CS}
 
It is sufficient to show that there exists a universal $\varepsilon_0>0$ such that for all $g \in L^2(\S^1)$, with $g \perp {\bf 1}$ and $\|g\|_{L^2(\S^1)} =1$, we have $\Psi({\bf 1} + \varepsilon g) \geq 0$ for $0 \leq \varepsilon < \varepsilon_0$. In order to simplify notation, let us write $g_i := g(\omega_i)$. Note that
$$\Psi({\bf 1} + \varepsilon g) = \varepsilon^2 \int_{(\S^1)^6} (g_1+ g_2 + g_3 - g_4 - g_5 - g_6)^2\, \big( |\omega_4 + \omega_5 + \omega_6|^2 -1 \big)\,\d\Sigma_{\vec{\omega}} + O(\varepsilon^3),$$
where the constant implicit in the big $O$ notation is uniform for $g$ satisfying $\|g\|_{L^2(\S^1)}\leq 1$. Let us investigate the second order term
\begin{align}\label{Evaluation_S}
\begin{split}
S & := \int_{(\S^1)^6} (g_1+ g_2 + g_3 - g_4 - g_5 - g_6)^2\, \big( |\omega_4 + \omega_5 + \omega_6|^2 -1 \big)\, \d\Sigma_{\vec{\omega}} \\
& = 6  \int_{(\S^1)^6} g_1^2\, \big( |\omega_4 + \omega_5 + \omega_6|^2 -1 \big)\, \d\Sigma_{\vec{\omega}} + 12  \int_{(\S^1)^6} g_1 g_2\, \big( |\omega_4 + \omega_5 + \omega_6|^2 -1 \big)\, \d\Sigma_{\vec{\omega}} \\
& \ \ \ \ \ \ \ \ \ \ \ \ \ \ \ \ \ - 18\int_{(\S^1)^6} g_1 g_4\, \big( |\omega_4 + \omega_5 + \omega_6|^2 -1 \big)\, \d\Sigma_{\vec{\omega}}\\
& = 12  \int_{(\S^1)^6} g_1^2\, \d\Sigma_{\vec{\omega}} -12  \int_{(\S^1)^6} g_1g_2\, \d\Sigma_{\vec{\omega}}+ 36  \int_{(\S^1)^6} g_1^2\,(\omega_4\cdot \omega_5)\, \d\Sigma_{\vec{\omega}} \\
&  \ \ \ \ \ \ \ \ \ \ \ \ \ \ \ + 36 \int_{(\S^1)^6} g_1g_2\,(\omega_4\cdot \omega_5)\, \d\Sigma_{\vec{\omega}} - 72 \int_{(\S^1)^6} g_1g_4\,(\omega_4\cdot \omega_5)\, \d\Sigma_{\vec{\omega}}\\
& =: 12 A - 12B +36 C + 36 D - 72 E.
\end{split}
\end{align}
By \eqref{I} we have (note that we are not assuming here that $g$ is even)
\begin{equation}\label{Evaluation_B}
B = (2\pi)^{5}\sum_{n\in\Z^\times} |\widehat{g}(n)|^2 \, (-1)^n \,\alpha_n,
\end{equation}
and similarly to \eqref{I} we obtain
\begin{align} \label{Evaluation_A}
A& = g^2\sigma * \sigma *\sigma *\sigma *\sigma *\sigma (0) = (2\pi)^{5} \alpha_0\, \widehat{g^2}(0) = (2\pi)^{5} \sum_{n\in\Z^\times} |\widehat{g}(n)|^2\, \alpha_0.
\end{align}
By \eqref{II} it follows that
\begin{equation}\label{Evaluation_D}
D = -(2\pi)^{5}\sum_{n\in\Z^\times} |\widehat{g}(n)|^2 \, (-1)^{n}\,\widetilde{\alpha}_n,
\end{equation}
and similarly to \eqref{II} we obtain
\begin{align}\label{Evaluation_C}
C & =(2\pi)^{-2}\int_{\R^2} \widehat{g^2\sigma}\,\widehat{\sigma}\,
\widehat{{{e_1}\sigma}}\,\widehat{{{e_{-1}}\sigma}}\,\widehat{\sigma}\,\widehat{\sigma}\,\dx = -(2\pi)^{5} \widetilde{\alpha}_0\, \widehat{g^2}(0) = -(2\pi)^{5} \sum_{n\in\Z^\times} |\widehat{g}(n)|^2\, \widetilde{\alpha}_0.
\end{align}
Finally, expanding the identity
\begin{equation*}
\int_{(\S^1)^6} g_1g_4\,|\omega_4 + \omega_5|^2\, \d\Sigma_{\vec{\omega}} = \int_{(\S^1)^6} g_1g_4\,|\omega_1 +\omega_2 +\omega_3 + \omega_6|^2\, \d\Sigma_{\vec{\omega}},
\end{equation*}
and using the symmetries to simplify,
we arrive at
\begin{equation}\label{Evaluation_E}
E = -\frac{B}{2} - \frac{3D}{2}.
\end{equation}
Combining \eqref{Evaluation_S}, \eqref{Evaluation_B}, \eqref{Evaluation_A}, \eqref{Evaluation_D}, \eqref{Evaluation_C} and \eqref{Evaluation_E} we obtain
\begin{align*}
S = 12(A +2B + 3C + 12D) = 12 \,(2\pi)^{5}\sum_{n\in\Z^\times} |\widehat{g}(n)|^2\,c_n,
\end{align*}
where
\begin{equation*}
c_n = \alpha_0 + 2 (-1)^n \alpha_n - 3 \widetilde{\alpha}_0 -12 (-1)^n \widetilde{\alpha}_n.
\end{equation*}
We must verify that $c_n >\eta> 0$ for all $n \in \Z^{\times}$, with $\eta$ universal. Since $c_n = c_{-n}$, we can restrict our attention to $n>0$. The cases $n=1,2,\ldots,6$ can be verified by direct computation using the values on Table \ref{table:T1}. 
For $n\geq 7$, we use Theorem \ref{alpha_n} to get
\begin{equation*}
\big| 6 \widetilde{\alpha}_n - \alpha_n\big| \leq \frac{3}{4 \pi^2 n} +\frac{21}{32 \pi^2 (n-1)n(n+1)} + \frac{7}{500 n^4} < 0.012
\end{equation*}
and hence
\begin{equation*}
c_n \geq \big(\alpha_0 - 3 \widetilde{\alpha}_0\big) - 2\big| 6 \widetilde{\alpha}_n - \alpha_n\big| > 0.134 - 0.024 >0.
\end{equation*}
This completes the proof of Theorem \ref{Prop12_local_CS}.

\smallskip

We note that Theorem \ref{Thm7} and Theorem \ref{Prop12_local_CS} provide an alternative proof of Theorem \ref{Thm1}.



\section*{Acknowledgements} 
The software {\it Mathematica} was used to compute the entries of Tables
\ref{table:T1} and \ref{table:T2}. We are grateful to Jon Bennett, Michael Christ, Ren\'e Quilodr\'an, Stefan Steinerberger and Po-Lam Yung for valuable discussions during the preparation of this work. Finally, we would like to thank HIM (Bonn), IMPA (Rio de Janeiro) and Universit\`a di Ferrara for supporting research visits.


\end{document}